\documentclass{article}

\usepackage{amsmath}
\usepackage{amssymb}
\usepackage{amsthm}
\usepackage[all]{xy}

\swapnumbers
\theoremstyle{plain}
\newtheorem{thm}{Theorem}[section]
\newtheorem{lem}[thm]{Lemma}
\newtheorem{cor}[thm]{Corollary}
\newtheorem{prop}[thm]{Proposition}

\theoremstyle{definition}

\newtheorem{rem}[thm]{Remark}
\newtheorem{dfn}[thm]{Definition}
\newtheorem{notation}[thm]{Notation}

\title{Rigidity theorem for presheaves with $\Omega$-transfers}

\date{}

\author{Alexander Neshitov}

\begin{document}

\maketitle

\begin{abstract}
In 1983 A.~Suslin proved the Quillen-Lichtenbaum conjecture about algebraic $K$-theory of algebraically closed fields. The proof was based on a theorem called the Suslin rigidity
theorem.
In the present paper we prove the rigidity theorem for homotopy
invariant presheaves with $\Omega$-transfers, introduced by
I.~Panin. This type of presheaves includes the $K$-functor and algebraic
cobordism of M.~Levine and F.~Morel.


\end{abstract}

\section{Introduction}
The original rigidity theorem of A.~Suslin can be formulated as
follows \cite{Suslin}:

Let $k_0\subseteq k$ be an extension of algebraically closed fields.
Then for every prime $p$ different from $\text{char } F$, the canonical map
\[
K_n(k_0,\mathbb{Z}/p)\to K_n(k,\mathbb{Z}/p)
\]
is an isomorphism, where $K_n$ is the $n$-th algebraic $K$-theory.
In later works of H.~Gillet, R.~Thomason~\cite{Gillet Thomason} and O.~Gabber~\cite{Gabber}
analogous results were obtained for the strict Henselization of a local
ring of a closed point on a smooth variety, and for Henselian pairs.
Further steps were made by I.~Panin, S.~Yagunov and J. Hornbostel (\cite{Panin Yagunov}, \cite{Yagunov}, \cite{Hornbostel Yagunov}), and O.~R\"ondings and P.~\O stv\ae r~\cite{Rondings
Ostvaer}.
All known proofs are based on the existence of transfers and on homotopy invariance of $K$-groups.
In the paper by A.~Suslin and V.~Voevodsky~\cite{Suslin Voevodsky}, $K$-groups were replaced by
algebraic singular homology groups. Their main result states that
for a variety over the complex numbers, the algebraic singular homology groups with
finite coefficients
coincide with usual singular homology.

The following conjecture about complex cobordism, due to V.~Voevodsky, is of special interest:
\[
MGL^{i,0}(\mathbb{C})/nMGL^{i,0}(\mathbb{C})=\Omega^U_{-2i}/n\Omega^U_{-2i},\quad
\]
where $MGL$ denotes the algebraic cobordism spectrum theory and $\Omega^U$ denotes the complex cobordism.

Recently G.~Garkusha and I.~Panin have developed an approach to solve
this conjecture, which requires a rigidity theorem for the so-
called homotopy-invariant presheaves $\mathcal{F}$ with $\Omega$-transfers, as well as
the fact that $\mathcal{F}(S)=\mathcal{F}(\text{Spec } k)$, where $S$ is a Henselization (at
a closed point) of a smooth variety over $k$.

In the present paper we prove the rigidity theorem for homotopy
invariant presheaves with $\Omega$-transfers. Namely, we prove the following

\begin{thm}
Let $k$ be an algebraically closed field, $y\in Y$ be a closed point
of a nonsingular variety over $k$,  and $S=\text{Spec } \mathcal{O}_{Y,y}^{et}$ a spectrum of a
local ring in \'etale topology.
Let $\mathfrak{X}/S$ be a smooth affine curve over $S$ which admits a good
compactification (see \ref{def_of_good_compactification}) and let $\mathcal{F}$ be a homotopy invariant
presheaf with $\Omega$-transfers such that $n\mathcal{F}=0$, where $(n,
\text{char } k)=1$.

Then for any two sections $x,y\colon S\to\mathfrak{X}$ coinciding in the
closed point of $S$,
the two maps $\mathcal{F}(x),\mathcal{F}(y)\colon \mathcal{F}(\mathfrak{X})\to\mathcal{F}(S)$ coincide.
\end{thm}
As a consequence we prove the following
\begin{cor}
For the closed point embedding $\imath\colon \text{Spec } k \to S$,
the induced pullback map $\mathcal{F}(\imath)\colon\mathcal{F}(S)\to\mathcal{F}(k)$ is an isomorphism.
\end{cor}

The present paper is organized as follows.
In section~\ref{presheavestr} we introduce the notion of a homotopy invariant presheaf
with $\Omega$-transfers.
In sections~\ref{generic position results} and~\ref{picard group} we provide preliminary results which will be used in
the proof of the main theorem. In section~\ref{main result} we prove the rigidity
theorem and its corollary. In the last section~\ref{examples} we provide non-trivial examples of
homotopy invariant presheaves with $\Omega$-transfers.

\section{Presheaves with $\Omega$-transfers}\label{presheavestr}

In the present section we introduce the notion of a presheaf with $\Omega$-transfers.
This will take several steps.

Consider the base field $k$.
Denote by $\text{\bf{Sch}}_k$ the category of separated schemes of finite type over $k$, and by $\text{\bf{Sm}}_k$ its full subcategory of smooth quasi-projective schemes over $k.$
\begin{dfn}\label{def_of_hat}
By an elementary triple $X\stackrel{f}\leftarrow Z\stackrel{g}\to Y$ we call a diagram
$$\xymatrix{
  &Z\ar[ld]_{f}\ar[rd]^{g} \\
  X && Y
}
$$
where $X,Y\in\text{\bf{Sm}}_k$ $Z\in\text{\bf{Sch}}_k$ is a scheme over $k$ and $f\colon Z\to X$
is finite flat surjective locally complete intersection (l.c.i) morphism of $k$-schemes and $g$ is an arbitrary morphism of $k$-schemes.
\end{dfn}

\begin{rem}\label{remark} Since $X$ is smooth over $k$, the structure morphism $Z\to\text{Spec } k$ is an l.c.i. morphism, so $Z$ is an l.c.i. scheme. This implies that $g$ is also an l.c.i.
morphism (\cite{Fulton}, B.7)
\end{rem}

\begin{dfn}\label{def_of_equiv_hats}
Two triples $X\stackrel{f}\leftarrow Z\stackrel{g}\to Y$ and $X\stackrel{f'}\leftarrow Z'\stackrel{g'}\to Y$ are equivalent, iff there is an isomorphism $\alpha\colon Z\to Z'$ such
that the following diagram commutes:
$$\xymatrix{
  &Z\ar[ld]_{f}\ar[rrd]^{g}\ar[r]^{\alpha}& Z'\ar[lld]_{f'}\ar[rd]^{g'} \\
  X &&& Y
}
$$
Denote by $T(X,Y)$ the set of equivalence classes of elementary triples $(X\leftarrow Z\to Y)$.
Further we will mean by a triple an element of $T(X,Y)$.
\end{dfn}
\begin{dfn}\label{def_of_composition}
Composition of triples $X\stackrel{f_1}\leftarrow Z_1\stackrel{g_1}\to Y$ and $Y\stackrel{f_2}\leftarrow Z_2\stackrel{g_2}\to W$ is defined as the triple
 $X\stackrel{f_1\circ\pi_1}\longleftarrow Z_1\times_Y Z_2\stackrel{g_2\circ\pi_2}\longrightarrow W$ where $\pi_i\colon Z_1\times_Y Z_2\to Z_i$ is the corresponding projection,
 $i=1,2.$
\end{dfn}

\begin{dfn}
A sum of two triples $X\stackrel{f}\leftarrow Z\stackrel{g}\to Y$ and $X\stackrel{f'}\leftarrow Z'\stackrel{g'}\to Y$ is defined as the triple $X\stackrel{f\coprod f'}\longleftarrow
Z\coprod Z'\stackrel{g\coprod g'}\longrightarrow Y$
\end{dfn}
Note that sum is compatible with  the composition of triples.
Note that $X\leftarrow\emptyset\to Y$ is a zero element and the set $T(X,Y)$ can be endowed with a structure of a commutative monoid.

\begin{dfn}
$\Omega$-category is a category $\mathcal{C}$ defined as follows:
$Ob \ \mathcal{C}$ are smooth varieties over $k$ and for any two smooth
varieties $X$ and $Y$ we define $Hom_{\mathcal{C}}(X,Y)=G(T(X,Y))$ to be the Grothendieck group of the commutative monoid $T(X,Y)$.
\end{dfn}
Note that $\Omega$-category $\mathcal{C}$ is a preadditive category with direct sum given by disjoint union, since $Hom_{\mathcal{C}}(X_1\coprod X_2,Y)=Hom_{\mathcal{C}}(X_1,Y)\oplus Hom_{\mathcal{C}}(X_2,Y).$

\begin{rem}
There is an inclusion $\text{\bf{Sm}}_k\to \mathcal{C}$ sending each morphism $f\colon X\to Y$ to the element of $Hom_{\mathcal{C}}(X,Y)$ defined by the triple $(X\stackrel{id}\leftarrow X\stackrel{f}\to Y).$
\end{rem}

\begin{dfn}\label{def_of_omega_presheaf}
A presheaf with $\Omega$-transfers (or, shorter, $\Omega$-presheaf) is an additive functor $\mathcal{F}\colon\mathcal{C}^{op}\to \text{\bf{Ab}}$ such that the canonical map $$\mathcal{F}(X\coprod
Y)\to\mathcal{F}(X)\oplus\mathcal{F}(Y)$$ is an isomorphism.
\end{dfn}

\begin{dfn}\label{def_of_presheaf_on_proobjects}
For any $\Omega$-presheaf $\mathcal{F}$ and any $\text{\bf{Sm}}_k$-pro-object $Y=\lim_i X_i$, $X_i\in \text{\bf{Sm}}_k$,
define $\mathcal{F}(Y)=\mathop{\rm colim}_i \mathcal{F}(X_i)$.
\end{dfn}

\begin{dfn}\label{def_of_homotopy_invariance}
We call $\Omega$-presheaf $\mathcal{F}$ homotopy invariant iff for any $X\in \text{\bf{Sm}}_k$
the canonical map $\mathcal{F}(p)\colon\mathcal{F}(X)\to\mathcal{F}(X\times\mathbb{A}^1_k)$ is an isomorphism, where $p\colon X\times\mathbb{A}^1_k\to X$ is the projection.
\end{dfn}

\begin{dfn}\label{def_of_good_compactification}
 Let $S$ be a local scheme, $\mathfrak{X}\to S$ be an element of $\text{\bf{Sm}}_S$ of relative dimension 1.
We will say that $\mathfrak{X}$ admits a good compactification, if there is an open embedding $\mathfrak{X}\subseteq\overline{\mathfrak{X}}$, where $\pi:\overline{\mathfrak{X}}\to S$ is a smooth $S$-projective scheme of relative dimension
1, and $\mathfrak{X}_{\infty}=\overline{\mathfrak{X}}\setminus\mathfrak{X}$ is finite and \'etale over $S.$
\end{dfn}

\section{Some generic position results}\label{generic position results}
In this section we provide some generic position results needed for the section~\ref{picard group}.
The results of this section are derived mostly from the Bertini theorem.

\begin{notation}\label{notation}
Throughout this section we will use the following notation:
\begin{itemize}
\item $k$ is an algebraically closed field,
\item $S=\text{Spec }\mathcal{O}$ is a regular Henselian local scheme over $k$ such that the residue field of $\mathcal{O}$ coincides with $k$,
\item $0$ denotes the closed point of $S$,
\item $\mathfrak{X}\in\text{\bf{Sm}}_S$ is a scheme of relative dimension 1, such that
  $\mathfrak{X}$ admits a good compactification $\pi:\overline{\mathfrak{X}}\to S$ (see~\ref{def_of_good_compactification}) and denote $\mathfrak{X}_{\infty}=\overline{\mathfrak{X}}\setminus\mathfrak{X}$,
\item $X=\mathfrak{X}\times_S 0$ and $\overline{X}=\overline{\mathfrak{X}}\times_S 0$.
\end{itemize}
\end{notation}

\begin{lem}\label{1}
For any $a\in\mathcal{O}$ and $f\in\mathcal{O}[z_1,\ldots z_n]$ denote by $\overline{a}$ the residue of $a$ in $k$ and by $\overline{f}$ the residue of $f$ in $k[z_1,\ldots z_n]$.
Let $x_1,\ldots x_k\colon S\to \mathbb{A}^n_S$ be $S$-points of $\mathbb{A}^n_S$ such that $\overline{x_1},\ldots, \overline{x_1}$ are pairwise distinct in $\mathbb{A}^n_k$, $f\in k[z_1,\ldots,z_n]$ and
$b_1,\ldots,b_k\in \mathcal{O}$ such that $f(\overline{x_i})=\overline{b_i}$ for all $1\leqslant i\leqslant k.$

Then there is $F\in \mathcal{O}[z_1,\ldots,z_n]$ such that $\overline{F}=f$ and $F(x_i)=b_i.$
\end{lem}
\begin{proof}
Write $x_i=(x_{i(1)} \ldots x_{i(n)}).$ Since all $\overline{x_i}$ are
pairwise distinct, we may assume that all $\overline{x_{i(1)}}$ are
pairwise distinct (otherwise we may change the coordinate system).
For a multi-index $\lambda=(\lambda_1,\ldots \lambda_n)$ denote by $z^{\lambda}=z_1^{\lambda_1}\ldots z_n^{\lambda_n}.$ Then we can write $f=\sum k_{\lambda}z^{\lambda}$.
Condition $f(x_i)=\overline{b_i}$ means that
$\sum k_{\lambda}\overline{x_i}^{\lambda}=\overline{b_i}$.

Let us find $a_{\lambda}$ in $\mathcal{O}$ such that
$\overline{a_{\lambda}}=k_{\lambda}$ for any multi-index $\lambda$ and
for any natural $i\leqslant k$ we have
$\sum a_{\lambda}x_i^{\lambda}=b_i$.
This is a linear system for $a_{\lambda}.$ We rewrite it in the
following way
$$a_{0\ldots 0}+a_{1\ldots 0}x_{i(1)}+\ldots + a_{k-1\ldots 0}x_{i(1)}^{k-1}=b_i-\sum_{\lambda\notin \{(j,\ldots 0)|0\leqslant j<k\}}a_{\lambda}x_i^{\lambda}\eqno{(1)}$$
Now for any $\lambda\notin \{(j,\ldots 0)\mid 0\leqslant j<k\}$ we set
$a_{\lambda}$ to be any lift of $k_{\lambda}.$ Thus we get a linear
system for $a_{0 \ldots 0}$, $a_{1,0 \ldots 0}$, $\ldots$, $a_{k-1,0\ldots 0}$
with the square matrix
$$A=\begin{pmatrix}
1,x_{1(1)},x_{1(1)}^2 \ldots x_{1(1)}^{k-1}\\
\ldots\\
1,x_{k(1)},x_{k(1)}^2 \ldots x_{k(1)}^{k-1}\\
\end{pmatrix}$$
This is Vandermonde matrix, $\det A=\prod_{i<j}(x_{i(1)}-x_{j(1)}).$
Since all $\overline{x_{i(1)}}$ are pairwise distinct, $\det A$ does not lie in the maximal ideal of local ring $\mathcal{O}.$
Then $A$ is invertible. Then there exist a unique solution $a_{0 \ldots 0},a_{1,0 \ldots 0} \ldots,a_{k-1,0\ldots 0}$ of the linear system $(1).$

\noindent Note that then $\overline{a_{0 \ldots 0}},\overline{a_{1,0 \ldots 0}}
\ldots,\overline{a_{k-1,0\ldots 0}}$ is a solution of the system
$$\overline{a_{0\ldots 0}}+\overline{a_{1\ldots 0}}~\overline{x_{i(1)}}+\ldots + \overline{a_{k-1\ldots 0}}~\overline{x_{i(1)}}^{k-1}=\overline{b_i}-\sum_{\lambda\notin \{(j,\ldots
0)|0\leqslant j<k\}}\overline{a_{\lambda}}~\overline{x_i}^{\lambda}$$ with invertible matrix, so it coincides with another solution
 $$\overline{a_{\scriptscriptstyle 0 \ldots 0}}=k_{\scriptscriptstyle
   0 \ldots 0},\; \overline{a_{\scriptscriptstyle 1,0 \ldots
     0}}=k_{\scriptscriptstyle 1,0 \ldots 0},\; \ldots,\; \overline{a_{\scriptscriptstyle k-1,0\ldots 0}}=k_{\scriptscriptstyle k-1,0\ldots 0}.$$
Thus, a polynomial $F=\sum a_{\lambda}z^{\lambda}$ satisfies the
desired properties.
\end{proof}

\begin{cor}\label{corollary1}
Let $f,g\in k[t_0,\ldots,t_n]$ and $G\in \mathcal{O}[t_0,\ldots,t_n]$ be homogeneous polynomials of degree $d$.
Suppose $g=\overline{G}$. Consider a set $x_1,\ldots,x_r\colon S\to\mathbb{P}^n_S$ of $S$-points in $\mathbb{P}^n_S$
and corresponding set $\overline{x_1},\ldots,\overline{x_r}\colon \text{Spec } k\to\mathbb{P}^n_k$ of
$k$-points in $\mathbb{P}^n_k$
Suppose $f(\overline{x_i})=g(\overline{x_i})$ for
all $1\leqslant i\leqslant r$.

Then there is a homogeneous polynomial $F\in\mathcal{O}[t_0,\ldots,t_n]$ of degree $d$ such that $\overline{F}=f$ and $F(x_i)=G(x_i)$.
\end{cor}
\begin{proof}
We can choose a system of homogeneous coordinates in $\mathbb{P}^n_S$, where $x_i=(1:a_1^{(i)}:\ldots:a_n^{(i)})$ for some $a_i^{(j)}\in\mathcal{O}$ for all $1\leqslant i\leqslant r.$
Note that then $\overline{x_i}=(1:\overline{a_1^{(i)}}:\ldots:\overline{a_n^{(i)}}), \ i=\overline{1..r}.$
By lemma \ref{1} there is  $F_a\in \mathcal{O}[t_1,\ldots,t_n]$ such that $\overline{F_a}(t_1,\ldots,t_n)=f(1,t_1,\ldots,t_n)$ and
$F_a(a_1^{(i)},\ldots,a_n^{(i)})=G(1,a_1^{(i)},\ldots,a_n^{(i)})$
Take $F\in \mathcal{O}[t_0,\ldots,t_n]$ to be the homogenization of $F_a.$ Then $F$ possesses the desired properties.
\end{proof}

\begin{lem}\label{general position lemma}
For an embedding $\overline{X}\hookrightarrow\mathbb{P}^n_k$ consider a set of closed
points $x_1$, $\ldots$, $x_r$, $y_1$, $\ldots$, $y_m\in \overline{X}$, and let $d\geqslant r$. Let
$h\in k[t_0,\ldots,t_n]$ be a homogeneous polynomial of degree $d$ such that $h(x_i)\neq0$ for all $i$.

Then there is a homogeneous polynomial $F$ of degree $d$ such that
\begin{itemize}
\item[(1)] $F(x_i)=h(x_i)$ for $i=1\ldots r$
\item[(2)] $F(y_j)\neq 0$ for $j=1\ldots m$
\item[(3)] $Z(F)$ intersects $\overline{X}$ transversally in $\mathbb{P}^n_k$.
\end{itemize}
\end{lem}

\begin{proof}
Denote by $V_d=\Gamma(\mathbb{P}^n_k,\mathcal{O}_{\mathbb{P}^n_k}(d))$ the space of degree $d$ homogeneous polynomials. Consider the set
$$W=\{F\in \mathbb{P}(V_d) \mid h(x_i)F(x_{i+1})-h(x_{i+1})F(x_i)=0, i=1,\ldots,(r-1) \}$$
Note that the condition $h(x_i)F(x_{i+1})-h(x_{i+1})F(x_i)=0$ is a linear homogeneous equation on coefficients of $F.$
So $W$ is a linear subspace of $\mathbb{P}(V_d)$ and its dimension
is at least $\dim\mathbb{P}(V_d)-r$.

Consider 3 sets:
\begin{itemize}
\item[] $U_1=\{F\in\mathbb{P}(V_d) \mid F(x_1)\neq 0,\ldots, F(x_r)\neq 0 \}$
\item[] $U_2=\{F\in\mathbb{P}(V_d) \mid Z(F) \text{ intersects $\overline{X}$ transversally } \}$
\item[] $U_3=\{F\in\mathbb{P}(V_d) \mid F(y_1)\neq 0,\ldots, F(y_m)\neq 0 \}.$
\end{itemize}
We will prove that $U_i$ is open in $\mathbb{P}(V_d)$ and $W\cap U_i\neq\emptyset$ for $i=1,2,3.$
Since $W$ is irreducible, any two nonempty open subsets intersect.
Then $W_0=(W\cap U_1)\cap (W\cap U_2)\cap (W\cap U_3)$ is nonempty.
Then take a polynomial $F_1$ that represents any element of $W_0$. Then there is a nonzero constant $c$ such that $F_1(x_i)=ch(x_i)$ for $i=1\ldots r$. Then the polynomial $F=F_1/c$
possesses the properties (1), (2), (3).

Now let us prove that $U_i$ is open in $\mathbb{P}(V_d)$ and $W\cap U_i\neq\emptyset$ for $i=1,2,3.$
\begin{itemize}
\item Since $U_1$ is a complement of $r$ hyperplanes, it is open in $\mathbb{P}(V_d)$. Since $h^d\in W\cap U_1$,
we have that $W\cap U_1$ is a nonempty open subset in $W$.

\item The Bertini theorem\cite[II.8.18]{Hartshorne} and the Veronese embedding imply that
$U_2$ is open in $\mathbb{P}(V_d)$.
To prove that $U_2\cap W\neq\emptyset$ we construct a set of linear homogeneous polynomials $L_i$ as follows:
By the Bertini theorem there is linear homogeneous polynomial $L_1$
such that
\[L_1(x_1)=0,\quad L_1(x_i)\neq 0\text{ for $i\neq 1$,}\quad Z(L_1) \text{ intersects $\overline{X}$ transversally.}\]
Take a linear polynomial $L_2$ such that
\begin{itemize}
\item
$L_2(x_2)=0$,
\item $L_2(x_i)\neq 0$ for $i\neq 2$,
\item
$L_2(y)\neq 0$ for any $y\in \overline{X}\cap Z(L_1)$,
\item
 $Z(L_2)$ intersects $\overline{X}$ transversally.
\end{itemize}
Iterating this procedure up to $x_r$ we obtain a set of polynomials $L_1,\ldots, L_r$.
Then we take a linear homogeneous polynomial $L_{r+1}$ such that $L_{r+1}(y)\neq 0$ for any $y\in (\overline{X}\cap Z(L_1))\cup \ldots\cup (\overline{X}\cap Z(L_r))$ and
$Z(L_{r+1})$ intersects $\overline{X}$ transversally.
Iterating this process up to $L_d$, we obtain a set of linear homogeneous polynomials $L_{r+1},\ldots, L_d$.
Now consider the polynomial $G=L_1L_2\ldots L_d$. By construction
$G(x_i)=0$ for every $i=1\ldots r$ so that $G\in W$.
Moreover, we have $\overline{X}\times_{\mathbb{P}^n_k}Z(G)=\coprod \overline{X}\times_{\mathbb{P}^n_k}Z(L_i)$ and  any closed set $Z(L_i)$ intersects $\overline{X}$ transversally, so $Z(G)$ intersects $\overline{X}$ transversally.
Therefore, $G\in W\cap U_2$.

\item Since $U_3$ is a complement of $m$ hyperplanes, it is open in $\mathbb{P}(V_d).$ Let us show that $W\cap U_3$ is nonempty. The Bertini theorem implies that there is a set of
    linear homogeneous polynomials $L_i$, $i=1\ldots d$ such that
$L_1(x_1)=0,\ldots, L_r(x_r)=0$ and $L_i(y_j)\neq 0$ for all $i$, $j$.
Then $G=L_1\ldots L_d\in W\cap U_3$.
\end{itemize}
\end{proof}

We will need the following special case of lemma~\ref{general position lemma}:
\begin{cor}\label{general position corollary}
Consider a closed embedding $\overline{X}\hookrightarrow\mathbb{P}^n_k$ and a set of closed points $x_1,\ldots,x_r\in \overline{X}$.
Consider a linear homogeneous polynomial $h\in k[t_0,\ldots,t_n]$ such
that $h(x_i)\neq 0$ for every $i=1\ldots r$.

Then for any
$d\geqslant r$ there is a homogeneous polynomial $F$ of degree $d$
such that
\begin{itemize}
\item $F(x_i)=h^d(x_i)$ for every $i=1\ldots r$
\item $Z(F)\cap Z(h)\cap \overline{X}=\emptyset$
\item $Z(F)$ intersects $\overline{X}$ transversally inside $\mathbb{P}^n_k$.
\end{itemize}
\end{cor}

\section{The Picard group of curve with good compactifications}\label{picard group}
This section is devoted to establish a presentation of the relative Picard group (see~\cite[\S~2]{Suslin Voevodsky}) of a curve $\mathfrak{X}$ admitting a good compactification
(see~\ref{def_of_good_compactification}).
This presentation is expressed in lemma~\ref{5}. As its consequence we derive the pairing lemma~\ref{6} that plays an important role in the proof of the main theorem. Throughout this
section we use the notation~\ref{notation}

\begin{lem}\label{1.7}
All closed points of the scheme $\overline{\mathfrak{X}}$ are contained in $\overline{X}.$
\end{lem}
\begin{proof}
Consider a closed point $z\in \overline{X}$. Since $\pi$ is projective, $\pi(z)$ is closed point of $S$. Then $z\in \overline{X}.$
\end{proof}
\begin{cor}\label{disjoint_subschemes}
Let $Z_1$ and $Z_2$ be closed subsets of $\overline{\mathfrak{X}}.$ Then $Z_1$ and $Z_2$ are disjoint if and only if $Z_1\cap \overline{X}$ and $Z_2\cap \overline{X}$ are disjoint
\end{cor}
\begin{proof}
Let $Z_1\cap \overline{X}$ and $Z_2\cap \overline{X}$ be disjoint. Suppose that $Z=Z_1\cap Z_2\neq\emptyset.$ Since $Z$ is nonempty, there is a closed point $z\in Z.$ Since $Z$ is closed, $z$ is a
closed point of $\mathfrak{X}.$ By lemma~\ref{1.7} $z\in \overline{X}.$ Then $z$ lies both in $Z_1\cap \overline{X}$ and $Z_2\cap \overline{X}.$
\end{proof}

\begin{lem}\label{lemma_regularity}
Consider a codimension 1 subscheme $Z\subseteq \mathfrak{X}$ such that $Z$ is closed in $\overline{\mathfrak{X}}$. Suppose that $Z$ intersects $X$
transversally.

Then $Z$ is regular at all closed points.
\end{lem}

\begin{proof}
Let $d$ denote the dimension of $S.$ For any closed point $z\in Z$ let
$p\colon \mathcal{O}_{X,z}\to\mathcal{O}_{Z,z}$ denote the canonical projection.
Lemma~\ref{1.7} implies that $z\in X.$ Let
$I_X$ and $I_Z$ denote the ideals defining $X$ and $Z$ in the local ring $\mathcal{O}_{\overline{\mathfrak{X}},z}.$

We prove that the local ring $\mathcal{O}_{Z,z}$ is regular.
It is sufficient to prove that  $$\dim_{\mathcal{O}_{Z,z}/\mathfrak{M}_{Z,z}}\mathfrak{M}_{Z,z}/\mathfrak{M}_{Z,z}^2=d=\dim Z.$$
Consider a homomorphism of $k$-vector spaces
\[\varphi\colon\mathfrak{M}_{\mathfrak{X},z}/\mathfrak{M}_{\mathfrak{X},z}^2\to \mathfrak{M}_{Z,z}/\mathfrak{M}_{Z,z}^2
\text{ defined by } x+\mathfrak{M}_{\mathfrak{X},z}^2\mapsto p(x)+\mathfrak{M}_{Z,z}^2.\]
This homomorphism is well defined since for every $y\in\mathfrak{M}_{\mathfrak{X},z}^2$
the element $p(y)$ lies in $\mathfrak{M}_{Z,z}^2$.
Observe that $\dim_k\mathfrak{M}_{Z,z}/\mathfrak{M}_{Z,z}^2=d$ and, hence, $\mathcal{O}_{Z,z}$
is regular,  follows from the following statements
\begin{itemize}
\item[(1)] the residue field $k(z)$ equals $k$
\item[(2)] $\varphi$ is surjective
\item[(3)] $\ker\varphi$ is nontrivial
\item[(4)] $\dim_k\mathfrak{M}_{\mathfrak{X},z}/\mathfrak{M}_{\mathfrak{X},z}^2=d+1$
\item[(5)] $\dim_k\mathfrak{M}_{Z,z}/\mathfrak{M}_{Z,z}^2\geqslant d$.
\end{itemize}

We prove (1)-(5) as follows:

\begin{itemize}
\item[(1)] Since $\mathfrak{X}\to S$ is a morphism of finite type, then $\text{Spec } k(z)\to \text{Spec } k$ is a morphism of finite type.
Then by Hilbert's Nullstellensatz $k(z)=k.$
\item[(2)] This follows from the fact that $p$ is surjective.
\item[(3)] Let us prove that $I_Z\nsubseteq\mathfrak{M}_{\mathfrak{X},z}^2.$
Since $Z$ intersects $X$ transversally, we have that $\mathfrak{M}_{\mathfrak{X},z}=I_Z+I_X.$
The ideal $I_X$ is generated by some elements $a_1,\ldots a_d\in I_X$, since $X$ is regular of codimension $d$~\cite[III.4.10]{Altman Kleiman}.
Then the inclusion $I_Z\subseteq\mathfrak{M}_{\mathfrak{X},z}^2$
would imply that $\mathfrak{M}_{\mathfrak{X},z}/\mathfrak{M}_{\mathfrak{X},z}^2$ is generated by elements $a_i+\mathfrak{M}_{\mathfrak{X},z}^2$, $i=1\ldots d$. This contradicts to the fact that $\dim\mathfrak{X}=d+1.$
So we have shown that there is $f\in I_Z\setminus \mathfrak{M}_{\mathfrak{X},z}^2.$ Then $f+\mathfrak{M}_{\mathfrak{X},z}^2$ is nonzero and lies in $\ker\varphi.$
\item[(4)] Since $\mathfrak{X}$ is regular, we have that $\dim_k \mathfrak{M}_{\mathfrak{X},z}/\mathfrak{M}_{\mathfrak{X},z}^2=\dim\mathfrak{X}=d+1$
\item[(5)] The fact that $k(z)=k$ and proposition
  \cite[III.4.7]{Altman Kleiman} imply that
$$\dim_k\mathfrak{M}_{Z,z}/\mathfrak{M}_{Z,z}^2\geqslant \dim Z=d. \qedhere$$
\end{itemize}
\end{proof}

\begin{lem}\label{2}
Consider a codimension 1 subscheme $Z\subseteq \mathfrak{X}$ such that $Z$ is closed in $\overline{\mathfrak{X}}$. Suppose that $Z$ intersects $X$
transversally.

Then $\pi|_Z\colon Z\to S$ is \'etale.
\end{lem}
\begin{proof} The proof consists of two steps. First, we check (Step 1)
  that
  $\pi$ is unramified at all closed points. Then we check (Step 2) that $\pi$
  is flat at all closed points. Since $Z\to S$ is locally of finite type, these two facts imply that $Z\to S$ is
  \'etale at all closed points and, hence, at some neighbourhood of
  closed points of $Z$. Therefore, $\pi\colon Z\to S$ is \'etale.

{\bf Step 1.} We prove that $Z\to S$ is unramified at any point of $Z\cap X$.
Denote by $0$ the closed point of the local scheme $S$. Consider $z\in Z\cap X$, i.e. $\pi(z)=0$.
Then we get a morphism $\pi^*\colon\mathcal{O}_{S,0}\to \mathcal{O}_{Z,z}$.
Let us show that $\mathfrak{M}_{Z,z}=\mathfrak{M}_{S,0}\mathcal{O}_{Z,z}$ in the local ring $\mathcal{O}_{Z,z}$. Since $Z$ intersects $X$ transversally, we have
$\mathfrak{M}_{\mathfrak{X},z}=I_Z+I_0$ in local ring $\mathcal{O}_{\mathfrak{X},z}$. Here $I_0=\mathfrak{M}_{S,0}\mathcal{O}_{\mathfrak{X},z}$ is the ideal defining $X$, and
$I_Z$ is the ideal defining $Z$.
Note that $\mathcal{O}_{Z,z}=\mathcal{O}_{\mathfrak{X},z}/I_Z$. Then $$\mathfrak{M}_{Z,z}=\mathfrak{M}_{\mathfrak{X},z}/I_Z=(I_Z+I_0)/I_Z=I_0/I_Z=\mathfrak{M}_{S,s}\mathcal{O}_{\mathfrak{X},\mathfrak{X}}/I_Z=\mathfrak{M}_{S,s}\mathcal{O}_{Z,z}.$$
Since $k$ is algebraically closed and $z$ is a closed point, we obtain
that $k(z)=k$ is a separable extension of $k(0)=k$.

{\bf Step 2.} We prove that $Z\to S$ is flat in any point of $Z\cap X$.
Lemma~\ref{lemma_regularity} implies that for any closed point $z$ the local ring $\mathcal{O}_{Z,z}$ is regular.
Let us check that $\mathcal{O}_{Z,z}$ is a quasifinite $\mathcal{O}_{S,0}$-module.
As it was shown in I, $\mathfrak{M}_{Z,z}=\mathfrak{M}_{S,0}\mathcal{O}_{Z,z}$. Hence,
$\mathcal{O}_{Z,z}/\mathfrak{M}_{S,0}\mathcal{O}_{Z,z}=\mathcal{O}_{Z,z}/\mathfrak{M}_{Z,z}=k$ is a finite dimensional vector space over $k(0)=k$.

Then the Grothendieck theorem~\cite[V.3.6]{Altman Kleiman} implies that $\mathcal{O}_{Z,z}$ is flat over $\mathcal{O}_{S,0}.$
Then $Z\to S$ is flat in all closed points of $Z.$ Then the map $\pi:Z\to S$ is flat in some open neighbourhood of closed points of $Z.$
Note that $Z$ is the only open neighbourhood of its closed points. Then $\pi:Z\to S$ is flat
\end{proof}

\begin{lem}\label{lemma_disjoint_union} Let $Z$ be a closed subscheme of codimension 1 in $\overline{\mathfrak{X}}$.  Suppose that $\pi|_Z:Z\to S$ is \'etale.
Then $Z=\coprod\limits_{i=1}^k s_i(S)$ for some sections $s_i:S\to \mathfrak{X},$ $i=1\ldots k.$
\end{lem}
\begin{proof}
Since $\pi|_Z$ is \'etale, the closed subscheme $\pi|_Z^{-1}(0)$ has dimension zero, so it consists of finite number of closed points $z_i$, $i=1\ldots k$
Then for any $i=1\ldots k$ there is a section of $s_i\colon S\to Z$ of $\pi|_Z$ such that $s_i(0)=z_i$~\cite[4.2 \S 4]{Milne}
Since $\pi|_Z$ is \'etale and $\pi|_Z\circ s_i=1_S,$ we have that $s_i$ is \'etale~\cite[VI 4.7]{Altman Kleiman}. Then $s_i:S\to Z$ is open map, so $s_i(S)$ is open in $Z.$
Since $s_i$ is a section of $\pi|_Z,$ hence a section of $\pi,$ $s_i(S)$ is closed in $\overline{\mathfrak{X}}.$

Let us check that $s_i(S)$ are disjoint.
Suppose the closed set $s_i(S)\cap s_j(S)$ is nonempty. Then there is a closed point $x\in s_i(S)\cap s_j(S).$ Then $x$ is a closed point of $\mathfrak{X}.$
Then $x\in X$ by lemma~\ref{1.7}. This is impossible for $i\neq j$ since $s_i(S)\cap X=z_i$ and $s_j(S)\cap X=z_j.$

We prove that $Z=\coprod s_i(S).$  Since all $s_i(S)$ are open in $Z$, we get that $Z~\setminus\coprod s_i(S)$ is a closed subset of $Z$ having no closed points.
Then $Z=\coprod s_i(S).$

\end{proof}
This lemma has two corollaries
\begin{cor}\label{cor_disjoint_union}
In conditions of lemma~\ref{2} the scheme $Z$ equals to a disjoint union of schemes $s_i(S)$ for some sections $s_i\colon S\to\mathfrak{X}.$
\end{cor}
\begin{proof}
Corollary follows directly from lemma~\ref{2} and lemma~\ref{lemma_disjoint_union}
\end{proof}
\begin{cor}\label{X infty structure}
The subscheme $\mathfrak{X}_{\infty}$ equals to the disjoint union $\coprod_{i=1}^{r} x_i(S)$ for some $S$-points $x_i:S\to\overline{\mathfrak{X}}.$
\end{cor}
\begin{proof}
By definition of good compactification(\ref{def_of_good_compactification}) the scheme $\mathfrak{X}_{\infty}$ is \'etale over $S$.
\end{proof}

\begin{notation}\label{homogeneous_map}
For any homogeneous polynomials $F_0, F_1\in\mathcal{O}[z_0\ldots z_n]$ of the same degree
consider the homogeneous graded ring homomorphism
$k[t_0,t_1]\to\mathcal{O}[z_0\ldots z_n]$ given by the rule $t_0\mapsto F_0$, $ t_1\mapsto F_1$
This gives us a rational map
\[
\alpha\colon\mathbb{P}^n_S=\text{\bf{ Proj }} \mathcal{O}[z_0\ldots z_n]\dasharrow \text{\bf{ Proj }} k[t_0,t_1]=\mathbb{P}^1_k
\]
This map is defined on the open subset $\mathbb{P}^n_S\setminus (Z(F_0)\cap Z(F_1)).$
We denote this subset by $U_{F_0,F_1}$ and the map $\alpha$ by $[F_0:F_1].$
\end{notation}

\begin{lem}\label{3}
Let $D\subseteq \mathfrak{X}$ be a very ample divisor on $\overline{\mathfrak{X}}$.

Then for all big enough $d$ there is a morphism $f\colon\overline{\mathfrak{X}}\to \mathbb{P}^1_k$
with the following properties:
\begin{itemize}
\item[(1)] $\mathfrak{X}_{\infty}\subseteq f^{-1}(1)$
\item[(2)] $div_0(f)=dD$
\item[(3)] Scheme-theoretic preimage $f^{-1}(\infty)$ intersects $\overline{X}$
  transversally.
\end{itemize}
\end{lem}
\begin{proof}
Consider an embedding $\alpha\colon\overline{\mathfrak{X}}\hookrightarrow\mathbb{P}^n_S$, such that $D$ is a hyperplane section $D=\overline{\mathfrak{X}}\cap Z(H)$ for some linear homogeneous polynomial
$H\in\mathcal{O}[x_0,\ldots,x_n].$
We will construct $f\colon\overline{\mathfrak{X}}\to\mathbb{P}^1_k$ in the form $f:\overline{\mathfrak{X}}\stackrel{\alpha}\hookrightarrow\mathbb{P}^n_S\stackrel{[F_0:F_1]}\dasharrow\mathbb{P}^1_k$
for specially chosen polynomials $F_0$ and $F_1$ such that $\overline{\mathfrak{X}}\subseteq U_{F_0,F_1}$.
Now we construct the homogeneous polynomials $F_0$ and $F_1.$

\noindent Since $\mathfrak{X}_{\infty}$ is \'etale over Henselian base $S$, by lemma~\ref{lemma_disjoint_union} it consists of finite number of $S$-points in $\mathbb{P}^n_S$.
Denote these points by $\xi_1,\ldots,\xi_r\colon S\to\overline{\mathfrak{X}}$, and by $x_1,\ldots,x_r\colon\text{Spec } k\to\overline{X}$ the corresponding $k$-points in $\overline{X}.$

\noindent Let $h=\overline{H}\in k[t_0,\ldots,t_n]$. Since $D$ does not intersect $\mathfrak{X}_{\infty}$, we have
$h(x_i)\neq 0$ for $1\leqslant i\leqslant r$. Corollary \ref{general position corollary} implies that for any large enough $d$ there is a homogeneous polynomial $f_1$ of degree $d$
such that
\begin{itemize}
\item $f_1(x_i)=h^d(x_i)$
\item $Z(f_1)\cap Z(h)\cap \overline{X}=\emptyset$
\item $Z(f_1)$ intersects $\overline{X}$ transversally.
\end{itemize}
Take $F_0$ to be equal to $H^d$ and $F_1$ to be a homogeneous polynomial $F_1\in \mathcal{O}[t_0,\ldots,t_n]$ such that $F_1(\xi_i)=H^d(\xi_i)$ and $\overline{F_1}=f_1$.
The polynomial $F_1$ exists by corollary~\ref{corollary1}.

First we check that $\overline{\mathfrak{X}}\subseteq U_{F_0,F_1}.$
Since $\overline{X}$ and $Z(h^d)\cap Z(f_1)$ are disjoint, the corollary~\ref{disjoint_subschemes} implies that $\overline{\mathfrak{X}}$ and $Z(F_0)\cap Z(F_1)$ are disjoint.
Then $\overline{\mathfrak{X}}\subseteq U_{F_0,F_1}.$ Denote $U_{F_0,F_1}$ by $U.$

So, the map $f=[F_0:F_1]\circ\alpha$ is well defined on $\overline{\mathfrak{X}}.$
Fix some system of homogeneous coordinates in $\mathbb{P}^n_k$ and take $1=[1:1]$, $0=[0:1]$ and $\infty=[1:0]$.

Now we check that $f$ satisfies the conditions $(1)-(3)$:
\begin{itemize}
\item[(1)] The statement follows from the fact that $F_1(\xi_i)=F_0(\xi_i)$
\item[(2)] The construction of $f$ implies that $div_0(f)=dZ(H)\cdot\overline{\mathfrak{X}}=dD$.
\item[(3)] Let us check that $f^{-1}(\infty)$ intersects $\overline{X}$ transversally.
To check that we will establish the following equality
$$\overline{X}\times_{\overline{\mathfrak{X}}}f^{-1}(\infty)=\overline{X}\times_{\mathbb{P}^n_k}Z(\overline{F_1}).\eqno{(*)}$$
Since $Z(\overline{F_1})$ intersects $\overline{X}$ transversally, the equality (*) implies that $f^{-1}(\infty)$ intersects $\overline{X}$ transversally in $\overline{\mathfrak{X}}.$

To establish the equality (*) we will prove the following chain of equalities:
$$\overline{X}\times_{\overline{\mathfrak{X}}}f^{-1}(\infty)\stackrel{(1)}=\overline{X}\times_{\mathbb{P}^1_k}\infty\stackrel{(2)}=\overline{X}\times_U Z(F_1)\cap
U\stackrel{(3)}=\overline{X}\times_{\mathbb{P}^n_S}Z(F_1)\stackrel{(4)}=\overline{X}\times_{\mathbb{P}^n_k}Z(\overline{F_1})$$

The equality (1) holds, since its both sides coincide with the edge of the following diagram consisting of two cartesian squares:
$$\xymatrix{
  ... \ar[rr]\ar[d]&&f^{-1}(\infty)\ar[rr]\ar[d] &&\{\infty\}\ar[d] \\
  \overline{X}\ar[rr]&& \overline{\mathfrak{X}}\ar[rr] && \mathbb{P}^1_k
}
$$
The equality (2) holds, since its both sides coincide with the edge of the following diagram consisting of two cartesian squares:
$$\xymatrix{
  ... \ar[rr]\ar[d]&&Z(F_1)\cap U\ar[rr]\ar[d] &&\{\infty\}\ar[d] \\
  \overline{X}\ar[rr]&& U\ar[rr] && \mathbb{P}^1_k
}
$$
The equality (3) follows from the fact that $U$ is an open neighbourhood of $\overline{X}$

\noindent The equality (4) holds, since its both sides coincide with the edge of the following diagram consisting of two cartesian squares:
$$\xymatrix{
  ... \ar[rr]\ar[d]&&Z(\overline{F_1}       )\ar[rr]\ar[d] &&Z(F_1)\ar[d] \\
  \overline{X}\ar[rr]&& \mathbb{P}^n_k\ar[rr] && \mathbb{P}^n_S
}\qedhere
$$
\end{itemize}
\end{proof}

\begin{lem}\label{4}
Let $f$ be a rational function on
$\overline{\mathfrak{X}}$ such that its divisor can be represented as $div(f)=\sum s_i(S) - \sum t_j(S)$ such that in the union $\{\overline{s_i(S)}\}\cup\{\overline{t_j(S)}\}$ each member is disjoint
from the others as subschemes of $\overline{X}$.

Then $f$ can be extended to a regular morphism $\overline{\mathfrak{X}}\to\mathbb{P}^1_S$ that is \'etale over $S\times 0$ and over $S\times \infty$.
\end{lem}
\begin{proof} {\bf I.} Let us prove that there is extension $F\colon\overline{\mathfrak{X}}\to\mathbb{P}^1_S$.
Since $\overline{s_i(S)}, \overline{t_j(S)}$ are distinct points of $\overline{X}$, corollary~\ref{cor_disjoint_union} implies that
$s_i(S)$ and $t_j(S)$ are disjoint in $\overline{\mathfrak{X}}$. Then
$\overline{\mathfrak{X}}$ is a vertex of the following couniversal square:
$$\xymatrix{
    \overline{\mathfrak{X}}\setminus \cup s_i(S) \setminus \cup t_j(S)\ar[rr]\ar[d] && \overline{\mathfrak{X}}\setminus\cup s_i(S)\ar[d]\\
    \overline{\mathfrak{X}}\setminus \cup t_j(S) \ar[rr] && \overline{\mathfrak{X}}
}
$$
We see that $\mathbb{P}^1_S$ is the vertex of the following couniversal square
$$\xymatrix{
    Spec \ \mathcal{O}[t,t^{-1}]\ar[rr]\ar[d] && Spec \ \mathcal{O}[t^{-1}]\ar[d]\\
    Spec \ \mathcal{O}[t] \ar[rr] && \mathbb{P}^1_S
}
$$
Rational function $f$ defines 2 regular morphisms
$f\colon\overline{\mathfrak{X}}\setminus \cup t_j(S)\to Spec \ \mathcal{O}[t]$
and $\frac{1}{f}\colon\overline{\mathfrak{X}}\setminus \cup s_i(S)\to Spec \ \mathcal{O}[t^{-1}]$.
Then, by the couniversal property, we get a regular morphism
$F\colon\overline{\mathfrak{X}}\to\mathbb{P}^1_S$
such that the following diagram commutes:
$$\xymatrix{
    \overline{\mathfrak{X}}\setminus \cup s_i(S) \setminus \cup t_j(S)\ar[rr]\ar[dd] && \overline{\mathfrak{X}}\setminus\cup s_i(S)\ar[dd]\ar[rd]^{\textstyle \frac{1}{f_1}} & \\
    & Spec \ \mathcal{O}[t,t^{-1}]\ar[rr]\ar[dd] && Spec \ \mathcal{O}[t^{-1}]\ar[dd]\\
    \overline{\mathfrak{X}}\setminus \cup t_j(S) \ar[rr]\ar[rd]^{\textstyle f_1} && \overline{\mathfrak{X}}\ar@{-->}[rd]^{\textstyle \hat{f_1}} &\\
    & Spec \ \mathcal{O}[t] \ar[rr] && \mathbb{P}^1_S
}
$$

{\bf II.} Let us check that the extension $F$ is \'etale over $0\times S$ and $\infty\times S$.
Note that $F$ is of finite type. It is sufficient to prove that $F$ is unramified and flat at any closed point $x$ such that $F(x)\in 0\times S$ or $F(x)\in\infty\times S.$

Consider $x\in X$ such that $F(x)\in 0\times S$. Then $x\in s_i(S)$ for some $i.$ We prove that $F$ is unramified at $x$.
There is the following diagram:
$$\xymatrix{
          \overline{\mathfrak{X}}\ar[rr]^{F}\ar@<0.5ex>[d]^{\pi}&&\mathbb{P}^1_S\ar@<0.5ex>[lld]^{p}\\
          S\ar@<0.5ex>[u]^{s_i}\ar@<0.5ex>[rru]^{0\times S}
}
$$
and corresponding diagram of local rings:
$$\xymatrix{
          \mathcal{O}_{\overline{\mathfrak{X}},x}\ar@<-0.5ex>[d]_{s_i^*}&&\mathcal{O}_{\mathbb{P}^1_S,0}\ar@<-0.5ex>[lld]_{0\times S^*}\ar[ll]_{F^*}\\
          \mathcal{O}_{S,0}\ar@<-0.5ex>[u]_{\pi^*}\ar@<-0.5ex>[rru]_{p^*}
}
$$
To check that $F$ is unramified at $x,$ we need to check that $$\mathfrak{M}_{\overline{\mathfrak{X}},x}=\mathfrak{M}_{\mathbb{P}^1_S,0}\cdot\mathcal{O}_{\overline{\mathfrak{X}},x}.\eqno (*)$$
where $\mathfrak{M}_{\overline{\mathfrak{X}},x}$ and $\mathfrak{M}_{\mathbb{P}^1_S,0}$ are maximal ideals in the local rings $\mathcal{O}_{\overline{\mathfrak{X}},x}$ and $\mathcal{O}_{\mathbb{P}^1_S,0}$
To verify this equality we will establish the following equalities:
\begin{itemize}
\item[(1)]$\mathfrak{M}_{\overline{\mathfrak{X}},x}= I_X+I_{s_i(S)}$
\item[(2)]$I_X=I_{\mathbb{P}^1_k}\cdot\mathcal{O}_{\overline{\mathfrak{X}},x}$
\item[(3)]$I_{s_i(S)}=I_{0\times S}\cdot\mathcal{O}_{\overline{\mathfrak{X}},x}.$
\end{itemize}
Where $I_X$ and $I_{s_i(S)}$ denote the ideals defining the closed subschemes $X$ and $s_i(S)$ in the local ring $\mathcal{O}_{\overline{\mathfrak{X}},x},$ and
$I_{\mathbb{P}^1_k}$ and $I_{0\times S}$ denote the ideals defining the closed subschemes $\mathbb{P}^1_k$ and $0\times S$ in the local ring $\mathcal{O}_{\mathbb{P}^1_S,0}.$

\begin{itemize}
\item
to check the equality (1) note that
$$s_i^*(I_X)=s_i^*(\pi^*(\mathfrak{M}_{S,0})\cdot\mathcal{O}_{\overline{\mathfrak{X}},x})=\mathfrak{M}_{S,0}\cdot s_i^*(\mathcal{O}_{\overline{\mathfrak{X}},x})=\mathfrak{M}_{S,0}\cdot \mathcal{O}_{S,0}=\mathfrak{M}_{S,0}.$$
Then $s_i^*(\mathfrak{M}_{\overline{\mathfrak{X}},x})\subseteq \mathfrak{M}_{S,0} = s_i^*(I_X).$ Then $\mathfrak{M}_{\overline{\mathfrak{X}},x}\subseteq I_X+\ker s_i^*.$ Note that $\ker s_i^*=I_{s_i(S)}$ and $I_X$ are contained in the maximal
ideal $\mathfrak{M}_{\overline{\mathfrak{X}},x}.$ Then we get the equality (2): $\mathfrak{M}_{\overline{\mathfrak{X}},x}\subseteq I_X+\ker s_i^*=I_X+I_{s_i(S)}$.
\item to check the equality (2) note that $X=0\times_S\overline{\mathfrak{X}}=\mathbb{P}^1_k\times_{\mathbb{P}^1_S}\overline{\mathfrak{X}}$.
\item to check the equality (3) note that $\coprod s_i(S)=\overline{\mathfrak{X}}\times_{\mathbb{P}^1_S}S\times 0,$ and $s_j(S)$ misses the neighbourhood of $x$ for $j\neq i.$
\end{itemize}

Since $S\times 0$ and $\mathbb{P}^1_k$ intersect transversally in $\mathbb{P}^1_S$, we get
$$\mathfrak{M}_{\mathbb{P}^1_S,0}=I_{\mathbb{P}^1_k}+I_{S\times 0}.$$
Multiplying this equality by $\mathcal{O}_{\overline{\mathfrak{X}},x}, $ we get equality (*):
$$\mathfrak{M}_{\mathbb{P}^1_S,0}\cdot\mathcal{O}_{\overline{\mathfrak{X}},x}=I_{\mathbb{P}^1_k}\cdot\mathcal{O}_{\overline{\mathfrak{X}},x}+I_{S\times 0}\cdot\mathcal{O}_{\overline{\mathfrak{X}},x}={(2)\& (3)}=I_X+I_{s_i(S)}=(1)=\mathfrak{M}_{\overline{\mathfrak{X}},x}.$$
So, $F$ is unramified at $x$.

Since $\mathcal{O}_{\overline{\mathfrak{X}},x}$ is a regular ring and quasi-finite $\mathcal{O}_{\mathbb{P}^1_S,0}$-module, the Grothendieck theorem \cite[V.3.6]{Altman Kleiman})
implies that $F$ is flat at $x$. Then
$F$ is \'etale at all closed points of $F^{-1}(0\times S)$. Then $F$ is \'etale over $0\times S$.
The same reasoning proves that $F$ is \'etale over ${S\times \infty}$. This concludes the proof.
\end{proof}

\begin{dfn}
A divisor $D$ on $\overline{\mathfrak{X}}$ is called an $S$-divisor iff its components are $S$-points of $\overline{\mathfrak{X}},$ i.e. $D=\sum_i n_is_i(S)$ for some sections $s_i\colon S\to\overline{\mathfrak{X}}.$
\end{dfn}

\begin{dfn}
A divisor $D$ on $\overline{\mathfrak{X}}$ is called a disjoint divisor, iff $D=\sum_i \varepsilon_iZ_i$ such that $\varepsilon_i\in\{-1,1\},$ and for $i\neq j$ we have that $Z_i\cap Z_j=\emptyset.$
\end{dfn}

\begin{dfn}
We will call the divisors $D$ and $D'$ disjoint, iff $D=\sum n_i Z_i$, $D'=\sum m_j Z'_j$ and for all $i,j$ the component $Z_i$ is disjoint from $Z'_j$.
\end{dfn}

\begin{lem}\label{lemma fixed point divisor}
Consider an $S$-point $s\colon S\to\mathfrak{X}.$
Consider a set of $S$-points $y_1,\ldots y_k\colon S\to\overline{\mathfrak{X}}.$
Then there is a rational function $\varphi$ on $\overline{\mathfrak{X}}$ such that

\begin{itemize}
\item $\varphi$ is defined and equal to $1$ on $\mathfrak{X}_{\infty}$,
\item $div\ \varphi=\sum s_i(S)-\sum t_m(S)$ is a disjoint $S$-divisor, $s_1=s$,
\item the divisor $\sum\limits_{i\geqslant 2} s_i(S)-\sum t_m(S)$ is disjoint from $\sum y_j(S)$.
\end{itemize}
\end{lem}
\begin{proof}
Consider a closed embedding $\overline{\mathfrak{X}}\hookrightarrow\mathbb{P}^n_S$. We will construct $\varphi$ in the form $\frac{F}{G}$, where $F$ and $G$ are in $\Gamma(\mathbb{P}^n_S, \mathcal{O}(d)$ for some d.
Corollary~\ref{X infty structure} implies that $\mathfrak{X}_{\infty}=\coprod x_i(S)$ for some set of $S$ points $x_i\colon S\to\overline{\mathfrak{X}},$ $i=1\ldots r.$
Denote by $\overline{x_i}$, $\overline{y_j}$, $\overline{s}$ the corresponding closed points $\text{Spec } k\to\overline{X}.$

{\bf Step 1.} We construct the polynomial $F$. For any $d\geqslant r$ by lemma~\ref{general position lemma} there is $f\in\Gamma(\mathbb{P}^n_k,\mathcal{O}(d))$ such that
\begin{itemize}
\item $f(\overline{s})=0$,
\item $f(\overline{y_j})\neq 0$ and $f(\overline{x_i})\neq 0$ for $j=1\ldots k,$ $i=1\ldots r.$
\item $Z(f)$ intersects $\overline{X}$ transversally in $\mathbb{P}^n_k.$
\end{itemize}
Then take $F\in \Gamma(\mathbb{P}^n_S, \mathcal{O}(d)$ such that $\overline{F}=f$ and $F(s)=0.$ This homogeneous polynomial exists by corollary~\ref{corollary1}

{\bf Step 2.} We construct the polynomial $G$. By lemma~\ref{general position lemma} there is a homogeneous polynomial $g\in k[z_0,\ldots,z_n]$ of degree $d$ such that:
\begin{itemize}
\item $g(y)\neq 0$ for all $y\in \overline{X}\cap Z(f)$, for $y=\overline{y_j}$ and for $y=\overline{x_i},$ $i=1\ldots r$ and $j=1\ldots k.$
\item $g(\overline{x_i})=f(\overline{x_i})$
\item $Z(g)$ intersects $\overline{X}$ transversally
\end{itemize}
By corollary~\ref{corollary1}, there is a homogeneous polynomial $G$ of degree $d$ such that
$\overline{G}=g$ and $G(x_i)=F(x_i)$. Then for all $x_i$ $\varphi$ is defined on $x_i$ and $\varphi(x_i)=1$

\noindent Now consider the divisor of $\varphi$ on $\overline{\mathfrak{X}}$

\noindent
Note that
\[
(\overline{\mathfrak{X}}\times_{\mathbb{P}^n_S}Z(F))\times_{\overline{\mathfrak{X}}}\overline{X}=Z(F)\times _{\mathbb{P}^n_S}\overline{X}=(Z(F)\times_{\mathbb{P}^n_S}\mathbb{P}^n_k)\times _{\mathbb{P}^n_k}\overline{X}=Z(f)\times _{\mathbb{P}^n_k}\overline{X}.
\]
Since $Z(f)$ intersects $\overline{X}$ transversally, we have that $\overline{\mathfrak{X}}\times_{\mathbb{P}^n_S} Z(F)$ intersects $\overline{X}$ transversally. Then by lemma \ref{2}, the scheme
$\varphi^{-1}(0)$ is \'etale over $S$. Then lemma~\ref{lemma_disjoint_union} implies that $\varphi^{-1}(0)=\coprod s_i(S)$ for some sections $s_i:S\to\mathfrak{X}.$
Since $F(s)=0$, we have that one of the $s_i,$ say $s_1$ equals $s.$ Then
 $div_0(\varphi)=\sum s_i(S)$.

The same reasoning proves that the scheme $Z(G)\times_{\mathbb{P}^n_S}\overline{\mathfrak{X}}$ intersects $\overline{X}$ transversally. Then it is \'etale over  $S$ by lemma~\ref{2}
Then by lemma~\ref{lemma_disjoint_union}  $Z(G)\times_{\mathbb{P}^n_S}\overline{\mathfrak{X}}=\coprod t_m(S)$.
Then $div_{\infty}(\varphi)=\sum t_m(S)$.

Then $div(\varphi)=div_0(\varphi)-div_{\infty}(\varphi)$ is a disjoint $S$-divisor.

By the construction of $\varphi$, for $i\geqslant 2$ $\overline{s_i}$ is distinct from all $\overline{y_j}$, and for all $m$ $\overline{t_m}$ is disjoint from all $\overline{y_j}$.
Then by corollary~\ref{cor_disjoint_union} the divisor $\sum_{i\geqslant 2}s_i(S)-\sum t_m(S)$ is disjoint from all $y_j(S)$.
\end{proof}

\begin{lem}\label{5}
Consider a map
\[
\Phi\colon\bigoplus_{s\colon S\rightarrow \mathfrak{X}}\mathbb{Z}\cdot s \longrightarrow Pic(\overline{\mathfrak{X}},\mathfrak{X}_{\infty})
\]
\[
s\mapsto s(S)
\]
Then
\begin{itemize}
\item $\Phi$ is surjective.
\item Kernel of $\Phi$ is generated by elements
$D=\sum s_i(S)-\sum t_j(S)$ for some finite sets of sections $\{s_i\}$ and $\{t_j\}$ of $\mathfrak{X}\to S$ such that there is a regular morphism
$f\colon\overline{\mathfrak{X}}\rightarrow \mathbb{P}^1_{S}$ with the following properties:
\begin{itemize}
\item[(i)] $f$ is \'etale over $0\times S$ and $\infty\times S$,
\item[(ii)] $\mathfrak{X}_{\infty}\subseteq f^{-1}(1\times S)$
\item[(iii)]$f^{-1}(0\times S)=\coprod s_i(S)$, $f^{-1}(\infty\times S)=\coprod t_j(S)$.
\end{itemize}
\end{itemize}

\end{lem}
\begin{proof}
(1) Let us prove that $\Phi$ is surjective.

(1.1) Consider $[Z]\in Pic(\overline{\mathfrak{X}},\mathfrak{X}_{\infty})$
such that $Z$ intersects $X$ transversally by finite number of points. By lemma \ref{2} $\pi\colon Z\to S$ is \'etale.
Then corollary~\ref{cor_disjoint_union} implies that $Z=\coprod s_i(S)$ for some set of sections $s_i$. Then $[Z]\in Im \Phi$.

(1.2) Now consider a case when $D=\sum n_i[Z_i]\in Pic(\overline{\mathfrak{X}},\mathfrak{X}_{\infty})$ is a very ample divisor.
We will prove that for big enough $d$, the divisor $dD$ is equivalent in
$Pic(\overline{\mathfrak{X}},\mathfrak{X}_{\infty})$ to some divisor $\sum n_i'Z_i'$ such that all $Z_i'$ intersect $X$ transversally.
By lemma \ref{3} there is a morphism $f\colon\overline{\mathfrak{X}}\to \mathbb{P}^1_k$ such that
\begin{itemize}
\item $\mathfrak{X}_{\infty}\subseteq f^{-1}(1)$
\item $div_0(f)=dD$
\item Scheme-theoretic preimage $f^{-1}(\infty)$ intersects $\overline{X}$ transversally.
\end{itemize}
Then the divisor $dD$ is equivalent to $div_{\infty}(f)$ in $Pic(\overline{\mathfrak{X}},\mathfrak{X}_{\infty})$. By the case 1.1, divisor $div_{\infty}(f)$ is in $Im \Phi$. The same reasoning shows that
$(d+1)D$ is in $Im \Phi$. Then $D=(d+1)D-dD$ is in $Im \Phi$. Since $Pic(\overline{\mathfrak{X}},\mathfrak{X}_{\infty})$ is generated by very ample divisors, we have that
$\Phi$ is surjective.

(2) Consider the kernel of $\Phi$.

\noindent For any divisor $D=\sum n_is_i(S)\in\ker\Phi$ the lemma~\ref{lemma fixed point divisor} implies that there is a disjoint $S$-divisor $D_1$ in $\ker\Phi$ such that
$D_1=s_1(S)+\sum s_i'(S) - \sum t_j'(S)$
and  the divisor $\sum s_i'(S) - \sum t_j'(S)$ is disjoint from $D.$

Then $$D-D_1=(n_1-1)s_1(S)+\sum_{i\geqslant 2}s_i(S)-\sum m_j t_j(S)-\sum s_i'(S)+\sum t_j'(S).$$
Note that we have reduced by 1 the multiplicity of $s_1(S)$ in $D$ and added a disjoint $S$-divisor that is disjoint from $D$.

Iterating this procedure we reduce all multiplicities $n_i$ and $m_j$ and on some step $k$ we will get $D-D_1-\ldots-D_k=\widetilde{D}$ where $\widetilde{D}\in\ker\Phi$ is
a disjoint $S$-divisor
Since $D_i$, $i=1\ldots k$ and $\widetilde{D}$ lie in $\ker\Phi,$ there are rational functions $f_i$, $i=1\ldots k$ and $\widetilde{f}$ such that $f_i(\mathfrak{X}_{\infty})=1,$
$\widetilde{f}(\mathfrak{X}_{\infty})=1$ and $div f_i=D_i$ and $div \widetilde{f}=\widetilde{D}.$ Then by lemma~\ref{4} each $f_i$ and $\widetilde{f}$ can be extended to the regular morphisms
$f_i,\widetilde{f}\colon\overline{\mathfrak{X}}\to\mathbb{P}^1_S$ possessing the properties (i)-(iii). Then $$D=D_1+\ldots D_k+\widetilde{D}$$ is the desired decomposition of $D$ in $\ker\Phi.$

\end{proof}

\begin{lem}\label{6}
Consider a pairing:
$$\bigoplus_{s\colon S\rightarrow \mathfrak{X}} \mathbb{Z}\cdot s\otimes \mathcal{F}(\mathfrak{X})\longrightarrow \mathcal{F}(S)\text{ defined by }s\otimes f\mapsto s^*(f).$$
Then this pairing can be decomposed in the following way:
$$\xymatrix{
    \bigoplus\limits_{s\colon S\rightarrow \mathfrak{X}} \mathbb{Z}\cdot s\otimes \mathcal{F}(\mathfrak{X}) \ar[rd]_{} \ar[rr]^{} &&  \mathcal{F}(S)\\
                        & Pic(\overline{\mathfrak{X}},\mathfrak{X}_{\infty})\otimes\mathcal{F}(\mathfrak{X})\ar[ru]
}
$$
\end{lem}
\begin{proof}
By lemma \ref{5}, it is sufficient to prove that for every
$f\colon\overline{\mathfrak{X}}\to \mathbb{P}^1_S$ such that:
\begin{itemize}
\item $\mathfrak{X}_{\infty}\subseteq f^{-1}(1\times S)$,
\item $f$ is \'etale over $0\times S$ and $\infty \times S$,
\item $f^{-1}(0\times S)=\coprod s_i(S)$,
\item $f^{-1}(\infty\times S)=\coprod t_j(S)$
\end{itemize}
The following equality holds:

$\sum \mathcal{F}(s_i) - \sum \mathcal{F}(t_j)=0$ as a morphism $\mathcal{F}(\mathfrak{X})\to\mathcal{F}(S) $

Note that the closed fiber restriction $f|_{\overline{X}}\colon\overline{X}\to\mathbb{P}^1_k$ is a quasifinite morphism.
Since $S$ is local, then $f\colon\overline{\mathfrak{X}}\to\mathbb{P}^1_S$ is quasifinite. Since $f$ is projective and quasifinite, then $f$ is finite. Take $\mathbb{A}^1_S=\mathbb{P}^1_S\backslash 1\times S$ and
consider $\mathfrak{X}_1=\overline{\mathfrak{X}} \times_{\mathbb{P}^1_S}\mathbb{A}^1_S$. Then $\mathfrak{X}_1$ is open subscheme in $\mathfrak{X}$, smooth over $S$ and $f|_{\mathfrak{X}_1}\colon\mathfrak{X}_1\to\mathbb{A}^1_S$ is finite and surjective.
Note that the following diagram commutes:
$$\xymatrix{
  \coprod s_i(S)\ar@{^{(}->}[rr]^{I_0}\ar[d]^{f|_{\coprod s_i(S)}} && \mathfrak{X}_1\ar[d]^{f|_{\mathfrak{X}_1}} && \coprod t_j(S)\ar@{_{(}->}[ll]_{I_\infty}\ar[d]^{f|_{\coprod t_j(S)}} \\
  S\ar@{^{(}->}[rr]^{i_0} && \mathbb{A}^1_S &&  S \ar@{_{(}->}[ll]_{i_\infty}
}
$$
All vertical arrows are finite surjective flat l.c.i. morphisms, so there are the following morphisms in the category $\mathcal{C}$:

\noindent $\phi\in Hom_{\mathcal{C}}(\mathbb{A}^1_S,\mathfrak{X}_1)$ represented by the triple $(\mathbb{A}^1_S\stackrel{}\leftarrow\mathfrak{X}_1\stackrel{id}\to\mathfrak{X}_1),$
$\alpha_0\in Hom_{\mathcal{C}}(S,\coprod s_i(S))$ represented by $(S\leftarrow\coprod s_i(S)\to\coprod s_i(S)),$ and $\alpha_1\in Hom_{\mathcal{C}}(S,\coprod t_j(S))$ represented by
$(S\leftarrow\coprod t_j(S)\to\coprod t_j(S)).$

Then we have the following diagram in $\Omega$-category:
$$\xymatrix{
  \coprod s_i(S)\ar@{^{(}->}[rr]^{I_0} && \mathfrak{X}_1 && \coprod t_j(S)\ar@{_{(}->}[ll]_{I_\infty} \\
  S\ar@{^{(}->}[rr]^{i_0}\ar^{\alpha_0}[u] && \mathbb{A}^1_S\ar^{\phi}[u] &&  S \ar@{_{(}->}[ll]_{i_\infty}\ar^{\alpha_{\infty}}[u]
}
$$
Let us prove that this diagram commutes:
\noindent The composition $I_0\circ\alpha_0$
is represented by the triple $(S\leftarrow\coprod s_i(S)\to \mathfrak{X}_1)$

\noindent The composition $\phi\circ i_0$ is represented by the triple $(S\leftarrow S\times_{\mathbb{A}^1_S}\mathfrak{X}_1\to \mathfrak{X}_1)$
Since $S \times_{\mathbb{A}^1_S}\mathfrak{X}_1=\coprod s_i(S)$, the left square commutes. Analogously, the right square commutes.
Applying the functor $\mathcal{F}$ we get the following commutative diagram:
$$\xymatrix{
  \mathcal{F}(\coprod s_i(S))\ar[d]^{\mathcal{F}(\alpha_0)} && \mathcal{F}(\mathfrak{X}_1)\ar[ll]_{\mathcal{F}(I_0)}\ar[rr]^{\mathcal{F}(I_{\infty})}\ar[d]^{\mathcal{F}(\phi)} && \mathcal{F}(\coprod t_j(S))\ar[d]^{\mathcal{F}(\alpha_{\infty})} \\
  \mathcal{F}(S) && \mathcal{F}(\mathbb{A}^1_S)\ar[rr]^{\mathcal{F}(i_\infty)}\ar[ll]_{\mathcal{F}(i_0)} &&  \mathcal{F}(S)
}
$$
Denote by $p$ the projection $p:\mathbb{A}^1_S\to S.$ The functor $\mathcal{F}$ is homotopy invariant and $i_0$ and $i_1$ are sections of $p.$ Then $\mathcal{F}(i_0)=\mathcal{F}(i_{\infty})=(\mathcal{F}(p))^{-1}.$
Then $$\mathcal{F}(\alpha_0)\circ \mathcal{F}(I_0)=\mathcal{F}(\alpha_{\infty})\circ\mathcal{F}(I_\infty)$$

\noindent We will prove that $\mathcal{F}(\alpha_0)\circ \mathcal{F}(I_0)=\sum \mathcal{F}(s_i)$ and $\mathcal{F}(\alpha_{\infty})\circ \mathcal{F}(I_{\infty})=\sum \mathcal{F}(t_j)$ in $Hom_{\text{\bf{Ab}}}(\mathcal{F}(\mathfrak{X}_1),\mathcal{F}(S))$

Consider $\alpha_{0,(i)}\in Hom_{\mathcal{C}}(S,\coprod s_i(S))$ represented by the triple
$(S\leftarrow s_i(S)\to\coprod s_i(S)).$ Then $$\alpha_0=\sum\alpha_{0,(i)}$$
Each morphism $I_0\circ\alpha_{0,(i)}$ is represented by the triple $(S\leftarrow s_i(S)\to\mathfrak{X}_1)$ Note that $\pi|_{s_i(S)}\colon s_i(S)\to S$ is an isomorphism with $s_i$ being is its
inverse. Then $$I_0\circ\alpha_{0,(i)}=s_i \text{ in } Hom_{\mathcal{C}}(S,\mathfrak{X}_1).$$
Then $$\mathcal{F}(\alpha_0)\circ \mathcal{F}(I_0)=\mathcal{F}(I_0\circ\alpha_0)=\mathcal{F}(\sum I_0\circ\alpha_{0,(i)})=\sum\mathcal{F}(I_0\circ\alpha_{0,(i)})=\sum\mathcal{F}(s_i)$$
The same reasoning proves that $\mathcal{F}(\alpha_{\infty})\circ \mathcal{F}(I_{\infty})=\sum\mathcal{F}(t_j)$.

Then $\sum\mathcal{F}(s_i)=\sum\mathcal{F}(t_j)\colon\mathcal{F}(\mathfrak{X}_1)\to\mathcal{F}(S)$. Taking composition with the inclusion $\mathfrak{X}_1\hookrightarrow \mathfrak{X}$, we get
\[\sum\mathcal{F}(s_i)=\sum\mathcal{F}(t_j)\colon\mathcal{F}(\mathfrak{X})\to\mathcal{F}(S).\qedhere\]
\end{proof}

\section{Proof of the main result}\label{main result}
\begin{thm}
Let $k$ be algebraically closed field, $S$ be a Henselization of a smooth variety over $k$ in a closed point, $\mathfrak{X}/S$
be smooth scheme over $S$ of relative dimension 1, admitting a good compactification.
Let $\mathcal{F}$ be a homotopy invariant presheaf with $\Omega$-transfers such that $n\mathcal{F}=0$, and $gcd(n,\text{char } k)=~1$.

\noindent Then for any two sections $x,y\colon S\to\mathfrak{X}$, such that $x$ and $y$ coincide in the closed point of $S$, two homomorphisms $\mathcal{F}(x),\mathcal{F}(y)\colon\mathcal{F}(\mathfrak{X})\to\mathcal{F}(S)$ coincide.
\end{thm}
\begin{proof}
Consider a pairing $\bigoplus\limits_{s\colon S\rightarrow \mathfrak{X}} \mathbb{Z}\cdot s\otimes \mathcal{F}(\mathfrak{X})\to\mathcal{F}(S)$,
defined by the rule
$$s~\otimes~\alpha\mapsto~\mathcal{F}(s)(\alpha).$$
Since $n\mathcal{F}=0$ and lemma \ref{6}, this pairing can be decomposed in the following way:
$$\xymatrix{
    \bigoplus\limits_{s\colon S\rightarrow \mathfrak{X}} \mathbb{Z}\cdot s\otimes \mathcal{F}(\mathfrak{X}) \ar[rd]_{} \ar[rr]^{} &&  \mathcal{F}(S)\\
                        & Pic(\overline{\mathfrak{X}},\mathfrak{X}_{\infty})/n\otimes \mathcal{F}(\mathfrak{X})\ar[ru]
}
$$
So, it is enough to prove that elements, defined by the divisors $x(S)$ and $y(S)$ coincide in $Pic(\overline{\mathfrak{X}},\mathfrak{X}_{\infty})/n$.
Here our proof is similar to the proof (4.3 \cite{Suslin Voevodsky}).
By  (2.2 \cite{Suslin Voevodsky}) $$Pic(\overline{\mathfrak{X}},\mathfrak{X}_{\infty})/n\subseteq H^2_{et}(\overline{\mathfrak{X}},j_{!}(\mu_n))$$
where $\mu_n$ is the group of $n$-roots of unity.
By proper base change theorem, $$H^2_{et}(\overline{\mathfrak{X}},j_{!}(\mu_n))=H^2_{et}(\overline{X},(j_0)_{!}(\mu_n))$$
where $\overline{X}$ and $X$ are fibers of $\overline{\mathfrak{X}}$ and $\mathfrak{X}$ over the closed point of $S$, and $j_0\colon X\hookrightarrow\overline{X}$ is the corresponding embedding.
Consider the commutative diagram:
$$\xymatrix{
        Pic(\overline{\mathfrak{X}},\mathfrak{X}_{\infty})/n\ar[rr]\ar[d] && H^2_{et}(\overline{\mathfrak{X}},j_{!}(\mu_n))\ar[d]\\
        Pic(\overline{X},X_{\infty})/n\ar[rr] && H^2_{et}(\overline{X},(j_0)_{!}(\mu_n))
}
$$
Then $Pic(\overline{\mathfrak{X}},\mathfrak{X}_{\infty})/n\subseteq Pic(\overline{X},X_{\infty})/n$.
Since $x$ and $y$ coincide in the closed point,  the images of $x(S)$ and $y(S)$ coincide in $Pic(\overline{X},X_{\infty})/n$. Then $x(S)$ and $y(S)$ coincide in
$Pic(\overline{\mathfrak{X}},\mathfrak{X}_{\infty})/n$.
This concludes the proof.
\end{proof}

\begin{cor}
Let $k$ be algebraically closed field, $Y$ be a smooth variety, $\mathcal{O}$ be a local ring of a closed point of $Y$ in \'etale topology and $S=\text{Spec } \mathcal{O}$. Let $\mathcal{F}$ be a homotopy invariant
presheaf with $\Omega$-transfers and $n\mathcal{F}=0$, where $(n, char \ k)=~1$. Then the homomorphism $\mathcal{F}(\imath)\colon\mathcal{F}(S)\to\mathcal{F}(k)$ is an isomorphism. where $\imath\colon\text{Spec } k \to S$
is the embedding of the closed point.
\end{cor}
\begin{proof}
Let $l$ denote the dimension of $Y$. Then $S$ is isomorphic to $S_l=\text{Spec } \mathcal{O}_{A_k^l,0}^{\text{\'et}}$. So, it is sufficient to prove that $\mathcal{F}(\imath)\colon\mathcal{F}(S_l)\to\mathcal{F}(\text{Spec } k)$ is
an isomorphism.
Consider the projection $\pi_l\colon S_l\to\text{Spec } k$.
Then we get that $\mathcal{F}(\imath)\circ\mathcal{F}(\pi_l)=id_{\mathcal{F}(\text{Spec } k)}$. Then $\mathcal{F}(\pi_l)$ is injective. Let us check that $\mathcal{F}(\pi_l)$ is surjective. We will prove it by induction on $l$.

\noindent The case $l=0$ is trivial.

\noindent We prove the induction step.
Since any \'etale neighbourhood of $0$ in $\mathbb{A}^l$ contains an affine \'etale neighbourhood, we have that $\mathcal{F}(S_l)=colim \ \mathcal{F}(W),$ where $W$ runs over all affine \'etale neighbourhoods
$(W,w_0)\to(\mathbb{A}^l,0).$
For any affine \'etale neighbourhood $W$ we will find a morphism $\pi:S_l\to S_{l-1}$ and a decomposition of the canonical map
$$\xymatrix{
\mathcal{F}(W)\ar[rr]^{can}\ar[dr] && \mathcal{F}(S_l)\\
& \mathcal{F}(S_{l-1})\ar[ur]_{\mathcal{F}(\pi)}
}\eqno{(*)}
$$
By the induction hypothesis, $\mathcal{F}(\text{Spec } k)\to\mathcal{F}(S_{l-1})$ is surjective, so this decomposition implies that the image of $\mathcal{F}(W)$ in $\mathcal{F}(S_l)$ is contained in the image of the
morphism $\mathcal{F}(\pi_l).$
Then $\mathcal{F}(\pi_l)$ is surjective.

Now we will construct the morphism $\pi$ and the decomposition $(*).$

\noindent By proposition~\cite[3.3]{Panin Stavrova Vavilov}
there is an open neighbourhood $U$ of $w_0$ in $W$, an open embedding $U\hookrightarrow \overline{U}$,
and the commutative diagram, where $Y$ denotes $\overline{U}\setminus U$
$$\xymatrix{
    U\ar[rr]^{open}\ar[drr]^{p} && \overline{U}\ar[d]^{\overline{p}} && Y\ar[ll]\ar[dll]_{q}\\
                        && U'
}
$$
such that
\begin{itemize}
\item $U'$ is an open subscheme of $\mathbb{P}^{l-1}_k$,
\item $\overline{p}$ is smooth projective of relative dimension 1
\item $q$ is \'etale and finite.
\end{itemize}

For any \'etale morphism $(V,v_0)\to (U',p(w_0))$ consider the fiber product diagram:

$$\xymatrix{
X\ar[rr]^{\varphi}\ar[d]^{\alpha} && S_l\ar[d]^{p\circ can}\\
         V\ar[rr]^{\text{\'et}} && U'
}
$$
Since $S_l$ is local, there is a unique closed point $x_0$ in $X$ lying over $v_0$ in $V$.
Since $\varphi$ is \'etale and $\varphi(x_0)=0$, lemma~\cite[4.2 \S 4]{Milne} implies that there is a section $\chi:S_l\to X$ sending $0$ to $x_0.$
Thus we get a morphism $\alpha\circ\chi\colon (S_l,0)\to (V,v_0).$
Since $S_{l-1}$ is a limit of a filtered system of \'etale neighbourhoods $(V,v_0)$ of $(A^l,0)$,
we obtain a morphism $\pi\colon S_l\to S_{l-1}$ and the following commutative diagram:
$$\xymatrix{
S_l\ar[rr]^{can}\ar[d]^{\pi} && U\ar[d]^{p}\\
S_{l-1}\ar[rr]^{can} && U'
}
$$

Denote by $f$ the composition $S_l\stackrel{\pi}\to S_{l-1}\to U'$.
Consider the fiber product:
$$\xymatrix{
   \mathfrak{X}\ar[rr]\ar[d] && U\ar[d]\\
         S_l\ar[rr]^{f} && U'
}
$$
We check that $\overline{\mathfrak{X}}=\overline{U}\times_{U'}S_l$ is a good compactification of $\mathfrak{X}$. Since
$\overline{\mathfrak{X}}\setminus\mathfrak{X}=(\overline{U}\setminus U)\times_{U'} S_l=Y\times_{U'} S_l$ and $Y$ is \'etale over $U'$, we have that
$\overline{\mathfrak{X}}\setminus\mathfrak{X}$ is \'etale over $S_l$.

\noindent So we can apply the rigidity theorem to $\mathfrak{X}$.
We construct two sections $s_1,s_2\colon S_l\to\mathfrak{X}$ as follows

First section, $s_1$ arises by universal property in the diagram:
$$\xymatrix{
  S_l\ar[rrrd]^{can}\ar@{=}[rdd]\ar@{-->}[rd]^{s_1}\\
   & \mathfrak{X}\ar[rr]\ar[d] && U\ar[d]\\
   &      S_l\ar[rr] && U'
}
$$
where $S_l\to U$ is the canonical map.

 To construct the second section, $s_2$, let us prove that there is a section
$$\xi\colon S_{l-1}\to S_l\text{ of the morphism }\pi\colon S_l\to S_{l-1}$$
By its construction $\pi$ is smooth of relative dimension 1. Then $\pi$ can be presented as a composition of \'etale morphism and a projection:
$$\pi\colon S_l\stackrel{\text{\'et}}{\rightarrow}S_{l-1}\times\mathbb{A}^1\to S_{l-1}$$
Consider the following Cartesian square:
$$\xymatrix{
 Z\ar[rr]^{g}\ar[d]^{r} && S_{l-1}\ar[d]^{i_0}\\
      S_l\ar[rr]^{et} && S_{l-1}\times\mathbb{A}^1
}
$$
where $i_0$ is an embedding as a zero section $S_{l-1}\times {0}$.
Since base change preserves \'etaleness, we have that $g\colon Z\to S_{l-1}$ is \'etale
Then by universal property of $S_{l-1}$ we get $h\colon S_{l-1}\to Z$ a section of $g$. Now take $\xi=r\circ h\colon S_{l-1}\to Z\to S_l$. Note that $\xi$ is a section of $\pi$.

\noindent Now define $s_2$ as a morphism given by the universal property in the diagram:
$$\xymatrix{
  S_l\ar[rrrd]^{f}\ar@{=}[rdd]\ar@{-->}[rd]^(.6){s_2}\\
   & \mathfrak{X}\ar[rr]\ar[d] && U\ar[d]\\
   &      S_l\ar[rr] && U'
}
$$
where $f\colon S_l\to U$ is defined as the composition $f\colon S_l\stackrel{\pi}{\rightarrow}S_{l-1}\stackrel{\xi}{\rightarrow}S_l\stackrel{can}{\rightarrow}U$.
It is easy to see that the sections $s_1$ and $s_2$ coincide in the closed point of $S_l$. Then theorem 1 implies that
$\mathcal{F}(s_1)=\mathcal{F}(s_2)\colon\mathcal{F}(\mathfrak{X})\to\mathcal{F}(S_l)$. Note that the image of the canonical morphism $\mathcal{F}(W)\to\mathcal{F}(S_l)$ equals to the composition:
$$\xymatrix{
   \mathcal{F}(W)\ar[r] & \mathcal{F}(U)\ar[r] & \mathcal{F}(\mathfrak{X})\ar[r]^{\mathcal{F}(s_1)} & \mathcal{F}(S_l).
}
$$
Since $\mathcal{F}(s_1)=\mathcal{F}(s_2)$, this composition equals
$$\xymatrix{
   \mathcal{F}(W)\ar[r] & \mathcal{F}(U)\ar[r] & \mathcal{F}(S_l)\ar[r]^{\mathcal{F}(\xi)}& \mathcal{F}(S_{l-1})\ar[r]^{\mathcal{F}(\pi)} & \mathcal{F}(S_l).
}
$$
So, the image of $\mathcal{F}(W)$ is contained in the image of $\mathcal{F}(S_{l-1})$. Then, by induction, it is contained in the image of $\mathcal{F}(Spec \ k)$.
So $\mathcal{F}(\pi_l)\colon\mathcal{F}(\text{Spec } k)\to \mathcal{F}(S_l)$ is an isomorphism. Consequently, $\mathcal{F}(\imath)\colon\mathcal{F}(S_l)\to\mathcal{F}(\text{Spec } k)$ is an isomorphism.
\end{proof}

\section{Examples of $\Omega$-presheaves}\label{examples}
In this section we prove that the K-functor and algebraic cobordism $\Omega^*$, defined by M. Levine and F. Morel in \cite{Levine Morel}
give the examples of $\Omega-$presheaves.

\paragraph{ K-theory.}

Consider a morphism $\alpha\in Hom_{\mathcal{C}}(X,Y)$ represented by the triple $(X\stackrel{f}\leftarrow Z\stackrel{g}\to Y)$
Since $Y\in\text{\bf{Sm}}_k$, we have that $g$ has finite Tor dimension and there is the pullback map $K'_*(Y)\to K'_*(Z)$. (\cite{Quillen}, 7.2.5) Since $f$ is finite, there is a pushforward map
$f_*\colon K'_*(Z)\to K'_*(X)$ (\cite{Quillen}, 7.2.7) Then we are able to give the following
\begin{dfn}
For the morphism $\alpha\in Hom_{\mathcal{C}}(X,Y)$ represented by the triple $(X\stackrel{f}\leftarrow Z\stackrel{g}\to Y)$ define $\alpha^*\colon K_*(Y)\to K_*(X)$ as the following
composition:
$$\alpha^*\colon K_*(Y)=K_*'(Y)\stackrel{g^*}\to K_*'(Z)\stackrel{f_*}\to K_*'(X)=K_*(X).$$
\end{dfn}

\begin{lem}
Let $\alpha\in\mathcal{C}(X,Y)$ and $\beta\in\mathcal{C}(Y,W)$. Then $(\beta\circ\alpha)^*=\beta^*\circ\alpha^*$.
\end{lem}
\begin{proof}
It is sufficient to prove the statement for $\alpha$ and $\beta$ defined by elementary triples.
Suppose $\alpha$ is defined by $(X\stackrel{f_1}\leftarrow Z_1\stackrel{g_1}\to Y)$ and $\beta$ is defined by $(Y\stackrel{f_2}\leftarrow Z_2\stackrel{g_2}\to W)$. We have a
commutative diagram defining the composition $\beta\circ\alpha$
$$\xymatrix{
  && Z_1\times_Y Z_2\ar[ld]_{p_1}\ar[rd]^{p_2} \\
  &Z_1\ar[ld]_{f_1}\ar[rd]^{g_1} && Z_2\ar[ld]_{f_2}\ar[rd]^{g_2} \\
  X && Y && W
}
$$
Since $f_2$ is a flat morphism, then $f_2\colon Z_2\to Y$ and $g_1\colon Z_1\to Y$ are Tor~-independent. Then by proposition 7.2.11 of \cite{Quillen} we have $g_1^*f_{2*}=h_{1*}h_2^*$.
Then
$$\beta^*\circ\alpha^*=f_{1*}g_1^*f_{2*}g_2^*=f_{1*}h_{1*}h_2^*g_2^*=(\beta\circ\alpha)^*.$$
\end{proof}
Since $K'_*$ is additive with respect to disjoint unions(\cite{Quillen},2.8), and homotopy invariant (\cite{Quillen},  7.4.1) we get
\begin{prop}
For any $X\in \text{\bf{Sm}}_k$, $\alpha\in Mor(\mathcal{C})$
the assignment $X\mapsto K_*(X)$, $\alpha\mapsto\alpha^*$ gives us the homotopy invariant presheaf with $\Omega$-transfers
$K_*\colon\mathcal{C}^{op}\to \text{\bf{Ab}}$

\end{prop}

\paragraph{ Algebraic cobordism.}

Consider a morphism $\alpha\in Hom_{\mathcal{C}}(X,Y)$ represented by the triple $(X\stackrel{f}\leftarrow Z\stackrel{g}\to Y)$.
Denote $\dim X=d_1$ and $\dim Y=d_2$.
In \cite{Levine Morel} there is constructed an extension of algebraic cobordism theory $\Omega^*$ to the Borel-Moore homology $\Omega_*$ on $\text{\bf{Sch}}_k$.

\noindent As we have mentioned in remark \ref{remark}, $g$ is an l.c.i. morphism.
Then by definition 6.5.10(\cite{Levine Morel}) there is the pullback $g^*\colon\Omega_*(Y)\to\Omega_{*+d_1-d_2}(Z)$, where Since $f$ is finite, there is the pushforward map
$f_*\colon\Omega_*(Z)\to\Omega_*(X)$. Since $X$ and $Y$ are smooth, we have that $\Omega_*(X)=\Omega^{d_1-*}(X)$ and $\Omega_*(Y)=\Omega^{d_2-*}(Y)$.

\begin{dfn}
For $\alpha\in \mathcal{C}(X,Y)$ defined by the triple $(X\stackrel{f}\leftarrow Z\stackrel{g}\to Y)$ define the pullback map $\alpha^*\colon\Omega^*(Y)\to\Omega^*(X)$ as the composition
$$\Omega^*(Y)=\Omega_{d_2-*}(Y)\stackrel{g^*}\to\Omega_{d_1-*}(Z)\stackrel{f_*}\to\Omega_{d_1-*}(X)=\Omega^*(X).$$
\end{dfn}

\begin{lem}
Let $\alpha\in\mathcal{C}(X,Y)$ and $\beta\in\mathcal{C}(Y,W)$. Then $(\beta\circ\alpha)^*=\beta^*\circ\alpha^*$.
\end{lem}
\begin{proof}
It is sufficient to prove the statement for $\alpha$ and $\beta$ defined by elementary triples.
Suppose $\alpha$ is defined by $(X\stackrel{f_1}\leftarrow Z_1\stackrel{g_1}\to Y)$ and $\beta$ is defined by $(Y\stackrel{f_2}\leftarrow Z_2\stackrel{g_2}\to W)$. We have a
commutative diagram defining the composition $\beta\circ\alpha$

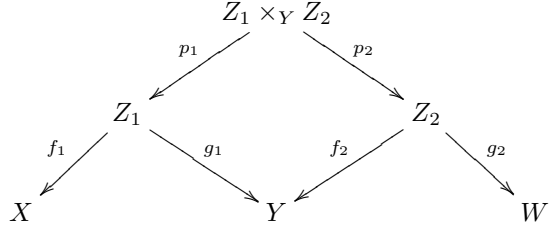
$$\xymatrix{
  && Z_1\times_Y Z_2\ar[ld]_{p_1}\ar[rd]^{p_2} \\
  &Z_1\ar[ld]_{f_1}\ar[rd]^{g_1} && Z_2\ar[ld]_{f_2}\ar[rd]^{g_2} \\
  X && Y && W
}
$$
Since $f_2$ is a flat morphism, then $f_2\colon Z_2\to Y$ and $g_1\colon Z_1\to Y$ are Tor~-independent. Then by theorem 6.5.12(\cite{Levine Morel}) we have $g_1^*f_{2*}=h_{1*}h_2^*$.
Then
$$\beta^*\circ\alpha^*=f_{1*}g_1^*f_{2*}g_2^*=f_{1*}h_{1*}h_2^*g_2^*=(\beta\circ\alpha)^*.$$
\end{proof}
Since $\Omega^*$ is additive with respect to disjoint unions(\cite{Levine Morel},2.4.13), and homotopy invariant (\cite{Levine Morel},  3.4.2) we get
\begin{prop}
For any $X\in \text{\bf{Sm}}_k$, $\alpha\in Mor(\mathcal{C})$
the assignment $X\mapsto\Omega^*(X)$, $\alpha\mapsto\alpha^*$ gives us the homotopy invariant presheaf with $\Omega$-transfers
$\Omega^*\colon\mathcal{C}^{op}\to~\text{\bf{Ab}}$

\end{prop}

\paragraph{Acknowledgments} I am grateful to Ivan Panin and Kirill
Zainoulline for constant attention and useful suggestions concerning
the subject of this paper. The research was supported by the Ontario
Trillium Scholarship, NSERC grant 396100-2010 and the RFBR grant 12-01-33057.

\bibliographystyle{plain}

\end{document}